\documentclass[12pt]{article} 

\usepackage{amsmath}
\usepackage{amsthm}
\usepackage{amsfonts}
\usepackage{mathrsfs}
\usepackage{stmaryrd}
\usepackage{setspace}
\usepackage{fullpage}
\usepackage{amssymb}
\usepackage{breqn}
\usepackage{enumitem}
\usepackage{bbold} 
\usepackage{authblk}
\usepackage{comment}
\usepackage{hyperref}
\usepackage{pgf,tikz}
\usepackage{graphicx}
\usepackage{subcaption}

\newtheorem{thm}{Theorem}[section]%[chapter]

\newtheorem{lemma}[thm]{Lemma}
\newtheorem{conjecture}[thm]{Conjecture}
\newtheorem{proposition}[thm]{Proposition}

\newtheorem{clm}[thm]{Claim}

\newcommand\ex{\ensuremath{\mathrm{ex}}}
\newcommand\rb{\ensuremath{\mathrm{rb}}}

\newcommand\cF{{\mathcal F}}

\newcommand\cH{{\mathcal H}}

\newcommand\cN{{\mathcal N}}

\newcommand{\ignore}[1]{}

\title{Rainbow copies of $F$ in families of $H$}
\author{Dániel Gerbner} 

\date{\small Alfr\'ed R\'enyi Institute of Mathematics\\
\small \texttt{gerbner.daniel@renyi.hu}}

\begin{document}

\maketitle

\begin{abstract}
    We study the following problem. How many distinct copies of $H$ can an $n$-vertex graph $G$ have, if $G$ does not contain a rainbow $F$, that is, a copy of $F$ where each edge is contained in a different copy of $H$? The case $H=K_r$ is equivalent to the Tur\'an problem for Berge hypergraphs, which has attracted several researchers recently. We also explore the connection of our problem to the so-called generalized Tur\'an problems. We obtain several exact results. In the particularly interesting symmetric case where $H=F$, we completely solve the case $F$ is the 3-edge path, and asymptitically solve the case $F$ is a book graph.
\end{abstract}

\section{Introduction}
Given a graph $G$ and a family $\cH$ of its subgraphs, we obtain an edge-coloring, where each subgraph in $\cH$ has its unique color and each edge of $G$ has a list of colors of all the subgraphs containing it. We say that a subgraph $F$ is \textit{rainbow} if we can pick a distinct color for each edge of $F$ from the corresponding list. Equivalently, $F$ is rainbow if there is an injection $\phi:E(G)\rightarrow\cH$ such that each edge is contained in its image. In other words, we pick the edges of $F$ from different members $\cH$. 

In this paper, $\cH$ will consist of copies of the same graph $H$.
Given $n$, $H$ and $F$, we are interested in the largest cardinality of a family $\cH$ of copies of $H$ in an $n$-vertex graph $G$ such that there is no rainbow $F$. We denote this number by
$\rb(n,H,F)$.

Let us mention some examples that (almost) fit our setting. Aharoni and Berger \cite{abe} showed that $2n-1$ mathings of size $n$ in bipartite graphs contain a rainbow matching of size $n$. Bar\'at, Gy\'arf\'as and S\'ark\"ozy \cite{bgs} studied how many matchings of size $m$ can be in a bipartite multigraph without a rainbow matching of size $m-k$. Aharoni, Briggs, Holzman and Jiang \cite{abhj} showed that every family of $2\lceil n/2\rceil-1$ odd cycles contain a rainbow odd cycle. Cheng, Han, Wang and Wang \cite{chww} studied the existence of rainbow $K_t$-factors. Note that in most of these examples the order of the graphs we count or forbid rainbow copies of increases with $n$.
Goorevitch and Holzman \cite{gh} considered fixed graphs and showed $\rb(n,K_3,K_3)=n^2/8$.

Observe that in most of the above examples, $H$ and $F$ was the same graph or the same family of graphs. Here we study the asymmetric version as well. Our main motivation to do that is that in the special case $H=K_r$, an equivalent problem has attracted a lot of attention in hypergraph Tur\'an theory.

Given a graph $F$, we say that a hypergraph $\cF$ is a \textit{Berge copy of $F$} (in short: \textit{Berge-$F$}) if there is a bijection $E(F)\rightarrow E(\cF)$ such that each edge is contained in its image. This is how Berge defined hypergraph cycles, and the definition was extended to arbitrary graphs by Gerbner and Palmer \cite{gp1}. Several papers studied the maximum number $\ex_r(n,\textup{Berge-}F)$ of hyperedges in an $r$-uniform Berge-$F$-free $n$-vertex hypergraph, see Subsection 5.2.2 of \cite{gp} for a slightly outdated survey.

Clearly we have that $\rb(n,K_r,F)=\ex_r(n,\textup{Berge-}F)$. One can see our problem as a generalization of Berge hypergraph where hyperedges are replaced by graphs.

Gy\H ori \cite{gyori} showed $\ex_3(n,\textup{Berge-}K_3)\le n^2/8$ if $n$ is large enough (implying most of the main result in \cite{gh}). Moreover, he studied multi-hypergraphs, thus his result imply the same bound if we allow a triangle to appear multiple times. This proves a conjecture from \cite{gh}.

Another related topic is generalized Tur\'an problems. We denote by $\cN(H,G)$ the number of copies of $H$ in $G$, and let $\ex(n,H,F)$ denote the largest $\cN(H,G)$ in $F$-free $n$-vertex graphs $G$. Note that $\ex(n,K_2,F)=\ex(n,F)$ is the ordinary Tur\'an problem. The systematic study of this generalized version was initiated by Alon and Shikhelman \cite{as} after several sporadic results, and has also attracted several researchers recently.

We can see our problem as a rainbow generalization of generalized Turán problems. Note that two different rainbow generalizations have already appeared in the literature \cite{gmmp,hp}, but there rainbowness comes from a proper edge-coloring, not the copies of $H$.

Clearly we have $\rb(n,H,F)\ge \ex(n,H,F)$, since an $F$-free graph cannot contain a rainbow $F$. We will show much stronger connection between the above two parameters.

We will often consider $F$ and $H$ fixed when $n$ grows. In particular, $o$ and $O$ are always used this way. 

The structure of the rest of the paper is the following. In Section 2, we generalize several known result on Berge hypergraphs to our setting. In particular, we present a connection to a variant of generalized Tur\'an numbers. In section 3, using the above mentioned connection, we show how some stabilty results on $\ex(,H,F)$ can imply to exact results on $\rb(n,H,F)$. In Section 4, we consider the symmetric case and determine $\rb(n,P_4,P_4)$ exactly. We also determine $\rb(n,B_t,B_t)$ asymptotically, where the \textit{book graph} $B_t$ consists of $t$ triangles sharing an edge. We finish the paper in Section 5 with some concluding remarks.

\section{Generalization of Berge results}

A simple technique often used in the theory of Berge hypergraphs is to pick the edges greedily. We show this by presenting the simplest example.

\begin{proposition}\label{gree}
    Let $G$ be a graph, $\cH$ a family of its subgraph, and $F$ be a subgraph such that each of its edges is in at least $|E(F)|$ elements of $\cH$. Then $F$ is rainbow.
\end{proposition}

\begin{proof}
    We go through the edges of $F$ in an arbitrary order. For each edge, we greedily pick a member of $\cH$ containing it that we have not picked earlier. It is doable, as even for the last edge, we can pick one of at least $|E(F)|$ edges, and at most $|E(F)|-1$ of them was picked earlier.
\end{proof}

The connection between Berge hypergraphs and generalized Tur\'an problems were established by Gerbner and Palmer \cite{gp2}, who proved that for any graph $F$, any $r$, and any $n$, we have 
\[
\ex(n,K_r,F)\le \ex_r(n,\textup{Berge-}F)\le \ex(n,K_r,F)+\ex(n,F).
\]
A strengthening of the result was proved by Gerbner, Methuku and Palmer \cite{gmp}. We say that a graph is \textit{red-blue} if each of its edges is colored red or blue. Given a red-blue graph $G$, we denote by $G_{red}$ and (respectively, $G_{blue}$) the subgraph spanned by the red (resp. blue) edges. Let $\cN^{{\mathrm col}}(H_1,H_2;G)=\cN(H_1,G_{red})+\cN(H_2,G_{blue})$, the number of red copies of $H_1$ and blue copies of $G_2$. Let $\ex^{{\mathrm col}}(n,H_1,H_2;F)$ denote the largest $\cN^{{\mathrm col}}(H_1,H_2;G)$ in $F$-free $n$-vertex red-blue graphs. Gerbner, Methuku and Palmer \cite{gmp} proved that $\ex_r(n,\textup{Berge-}F)\le \ex^{{\mathrm col}}(n,K_r,K_2;F)$. Note that an equivalent statement was proved by  F\"uredi, Kostochka, and Luo~\cite{fkl}.

We can generalize both above results to our setting. Clearly the second one implies the first, but we present a proof of the first as well, since it is simple and makes the second proof easier to understand.

\begin{lemma}\label{celeb}
    For any graphs $H$ and $F$ and for any $n$, we have $\ex(n,H,F)\le \rb(n,H,F)\le \ex(n,H,F)+\ex(n,F)$.
\end{lemma}

\begin{proof}
    To prove the first inequality, we take all the $\ex(n,H,F)$ copies of $H$ in an $F$-free $n$-vertex graph. Clearly there is no rainbow copy of $F$ there. To prove the second inequality, we take an $n$-vertex graph $G$ and a family $\cH$ of copies of $H$ in $G$. We go through the members of $\cH$ in an arbitrary order, and take a previously not picked edge from each of them, if possible. If not possible (because all the edges of the copy of $H$ have been picked earlier), then we mark the copy of $H$. The edges picked form an $F$-free graph $G'$. The copies of $H$ we could not pick an edge from are each subgraphs of $G'$. This shows that we could pick an edge at most $\ex(n,F)$ times and we could not pick an edge at most $\ex(n,H,F)$ times.
\end{proof}

Observe that in the above proof, we picked a maximal matching between the copies of $H$ in $G$ and edges of $G$ greedily. To improve this, we will pick a largest matching.
We follow the formulation of the proof as in \cite{gerbner}. In particular, we use the following lemma.

\begin{lemma}\label{cel2}[\cite{gerbner}] Let $\Gamma$ be a finite bipartite graph with parts $A$ and $B$ and let $M$ be a largest matching in $\Gamma$. Let $B'$ denote the set of vertices in $B$ that are incident to $M$. Then we can partition $A$ into $A_1$ and $A_2$ and partition $B'$ into $B_1$ and $B_2$ such that for $a\in A_1$ we have $M(a)\in B_1$, and every neighbor of the vertices of $A_2$ is in $B_2$. 
%Moreover, every vertex in $A_1$ has a neighbor in $B\setminus B'$.
\end{lemma}

\begin{lemma}\label{celeb2}
    For any graphs $H$ and $F$ and for any $n$, we have $\rb(n,H,F)\le \ex^{{\mathrm col}}(n,H,K_2;F)$.
\end{lemma}

\begin{proof}
    Let $G$ be an $n$-vertex graph, $\cH$ be a collection of copies of $H$ in $G$ such that $G$ is rainbow $F$-free. Let part $A$ of $\Gamma$ be $\cH$ and part $B$ be $E(G)$, and we join $a\in A$ to $b\in B$ if $a$ contains $b$. Now we apply Lemma \ref{cel2} to $\Gamma$ and an arbitrary largest matching $M$. The elements of $B'$ form an $F$-free graph $G'$. We color the edges in $B_1$ blue and edges of $B_2$ red. We have $|\cH|=|A_1|+|A_2|=|B_1|+|A_2|=|E(G_{blue})|+|A_2|$. As the copies of $H$ in $A_2$ have all their neighbors in $B_2$, they are red, showing $|A_2|\le N(K_r,G_{red})$.
\end{proof}

Let us remark that $\ex^{{\mathrm col}}(n,H_1,H_2;F)$ (and more generally an $r$-colored version) was studied in \cite{gerbner2}.

The Ramsey number $R(H,G)$ is the smallest integer $r$ such that in any red-blue coloring of $K_r$ there is either a red $H$ or a blue $G$. Let $F\setminus e$ denote a graph obtained by deleting an arbitrary edge $e$ from $F$.
Gr\'osz, Methuku and Tompkins \cite{GMT} proved that if $r\ge R(F,F\setminus e)$, then $\ex_r(n,\textup{Berge-}F)=o(n^2)$. This generalizes to our setting as follows.

\begin{proposition}
\textbf{(i)}    Let us assume that $H$ contains $F$. Then $\rb(n,H,F)=O(n^2)$.

\textbf{(ii)}    Let us assume that in any red-blue coloring of $H$, there is either a monored $F$ or a monoblue $F\setminus e$ for some edge $e$ of $F$. Then $\rb(n,H,F)=o(n^2)$.
\end{proposition}

\begin{proof}
 Observe that \textbf{(i)} is a simple corollary of Lemma \ref{celeb}. To prove \textbf{(ii)}, we pick an element of $\cH$ and we denote it by $H_0$. Let us color the edges of $H_0$ red if they are contained in less than $|E(F)|$ elements of $\cH$ and blue otherwise. 
 
 We claim that there is no blue $F\setminus e$. Indeed, otherwise we can find a rainbow $F\setminus e$ in $\cH\setminus \{H_0\}$ by Proposition \ref{gree}, and this is extended to a rainbow $F$ using $H_0$ for the edge $e$.
 This implies that $H_0$ contains a red $F$. 
 
 Let us bound the number of copies of $F$ in $G$. Each copy of $F$ contains at least two edges, thus at least three vertices from an element of $\cH$. This means we can pick every $F$ the following way. We pick an element of $\cH$ ($O(n^2)$ ways), three vertices of that element ($O(1)$ ways), $|V(F)|-3$ arbitrary other elements $(O(n^{|V(F)|-3})$ ways), and then a copy of $F$ on the $|V(F)|$ vertices ($O(1)$ ways). Therefore, there are $O(n^{|V(F)|-1})$ copies of $F$ in $G$.

 By the removal lemma, there is a set $E_0$ of $o(n^2)$ edges such that each copy of $F$ contains an element of $E_0$. As each element of $\cH$ contains a red $F$, it also contains a red edge in $E_0$. There are $o(n^2)$ red edges in $E_0$ and each are contained in at most $|E(F)|-1$ elements of $\cH$, thus there are $o(n^2)$ elements in $\cH$.
\end{proof}

\section{Proofs using stability}

In this section we prove exact results on $\rb(n,H,F)$ using stability results on $\ex(n,H,F)$. Most of our statements will actually show that $\ex^{col}(n,H,K_2:F)=\ex(n,H,F)$.

Most stability results on $\ex(n,H,F)$ belong to one of two types. In the first type we know an extremal graph $G$ with $\cN(H,G)=\ex(n,H,F)$ (or a family of extremal graphs), and we know that if an $F$-free graph $G'$ is not a subgraph of $G$, then $G$ contains $\cN(H,G)-\Omega(n^{|V(H)|-1})$ copies of $H$. This immediately implies that $\rb(n,H,F)=\ex^{col}(n,H,K_2:F)=\ex(n,H,F)$ in conjunction with Lemma \ref{celeb2}, if $|V(H)|\ge 4$. Such stability result can be found e.g. in \cite{gepat} concerning $\ex(n,K_{a,b},K_{s,t})$ for some values of $a,b,s,t$.

The other type of stability results is a generalization of the Erd\H os-Simonovits stability \cite{erd1,erd2,sim} for ordinary Tur\'an problems. The \textit{edit distance} of two graphs $G$ and $G'$ on the same vertex set is the smallest number of edges that we can add and delete from $G$ to obtain $G'$. Given graphs $H$ and $F$ with $\chi(H)<\chi(F)$, we say that $H$ is \textit{$F$-Tur\'an-stable} if for any $n$-vertex graph $G$ with $\cN(H,G)\ge \ex(n,H,F)-o(n^2)$, $G$ has edit distance $o(n^2)$ from the Tur\'an graph $T(n,r)$. We say that $H$ is \textit{weakly $F$-Tur\'an-stable} if for any $n$-vertex graph $G$ with $\cN(H,G)\ge \ex(n,H,F)-o(n^2)$, $G$ has edit distance $o(n^2)$ from a complete $(\chi(F)-1)$-partite graph.

The first such stability result in generalized Tur\'an problems is due to Ma and Qiu \cite{mq} who showed that $K_r$ is $F$-Tur\'an-stable for every $F$ with chromatic number more than $r$. Other results appear in \cite{hhl,gerb2,gerb3,lm}. These results were used to prove sharp bounds on $\ex(n,H,F)$. In particular the author showed in \cite{gerb2} that if $H$ is weakly $F$-Tur\'an-stable and $F$ has a color-critical edge (an edge whose deletion decreases the chromatic number), then $\ex(n,H,F)=\cN(H,T)$ for a complete $(\chi(F)-1)$-partite graph $T$. Some exact results were obtained in \cite{gerb3} in the case $F$ does not have a color-critical edge.

\begin{thm}
    If $H$ is weakly $F$-Tur\'an-stable and $F$ has a color-critical edge, then $\rb((n,H,F)=\ex(n,H,F)=\ex^{{\mathrm col}}(n,H,K_2;F)=\cN(H,T)$ for a complete $(\chi(F)-1)$-partite graph $T$.
\end{thm}

\begin{proof}
    Let $G$ be an $n$-vertex graph such that $G$ is $F$-free and  $\cN^{{\mathrm col}}(H,K_2;G)=\ex^{{\mathrm col}}(n,H,K_2;F)$. If $\cN(H,G)=\ex(n,H,F)-\Omega(n^{|V(H)|}$, then $\cN^{{\mathrm col}}(H,K_2;G)<\ex(n,H,F)\le \ex^{{\mathrm col}}(n,H,K_2;F)$, a contradiction. Therefore, by the weak Tur\'an stability, we have that $G$ has edit distance $o(n^2)$ from a complete $(\chi(F)-1)$-partite graph $T$. Let $A_i$ denote the parts of $T$. A lemma in \cite{gerb2} ensures that each part $A_i$ has order $\Theta(n)$.
    
    Assume first that there are $x$ blue edges between the parts of $T$. Then, compared to the monored $T$, we lose $\Omega(n^{|V(H)|-2})$ red copies of $H$ for each such edge. As each copy of $H$ is counted $O(1)$ times, we lose $\Omega(xn^{|V(H)|-2})$ red copies of $H$. We are done if there is no blue edge inside any part.
    %Since there are $x+o(n^2)$ blue edges, we obtain a contradiction if $|V(H)|\ge 4$. Moreover, if $|V(H)|=3$ ???...

    Assume now that there is a blue edge $uv$ inside a part, say, $A_1$. Then we claim that $\Omega(n)$ edges are missing between parts. Indeed, let us pick $u,v$ and $|V(F)|$ vertices from $A_1$. Then we pick $|V(F)|$ vertices from their common neighborhood in $A_2$ if possible. We continue this way and for each $i$, we pick $|V(F)|$ vertices in $A_i$ that are in the common neighborhood of the vertices picked earlier. If it is always possible, the resulting subgraph contains a copy of $F$, a contradiction. Thus at one point the vertices picked earlier have less than $|V(F)|$ common neighbors in $A_i$, thus the other $\Omega(n)$ vertices of $A_i$ each have a non-neighbor in the other parts.

    This implies that compared to the monored $T$, we lose $\Omega(n^{|V(H)|-1})$ copies of $H$. If $|V(H)|=2$, the statement is trivial. If $|V(H)|\ge 3$, then we have $x+o(n^2)$ blue edges and we lose $\Omega(xn+n^2)$ red copies of $H$, completing the proof.    
    %... nagy matching vagy star, vagy elég hogy lin él hiányzik ha $y>0$?
\end{proof}

Note that the above theorem determines $\ex_r(n,\textup{Berge-}F)$ exactly for sufficiently large $n$ if $F$ has chromatic number more than 3 and a color-critical edge. This is already known, but no simple proof exists.
The $r$-uniform \textit{expansion} $F^{+r}$ of a graph $F$ is the specific $r$-uniform Berge copy that contains the most vertices, i.e. the $r-2$ vertices added to each edge of $F$ are distinct for different edges, and distinct from the vertices of $F$. Pikhurko \cite{pikhu} determined the Tur\'an number of $K_k^{+r}$ if $k>r$. According to the survey \cite{mubver} on expansions, Alon and Pikhurko observed that Pikhurko's proof generalizes to the case $F$ is a $k$-chromatic graph with a color-critical edge. This gives the exact value of $\ex_r(n,\textup{Berge-}F)$ for sufficiently large $n$, but the proof is complicated.

We remark that $\ex(n,H,F)$ is not necessarily equal to $\ex^{{\mathrm col}}(n,H,K_2;F)$ even if $H$ is weakly $F$-Tur\'an-stable.
Let $F_2$ denote two triangles sharing a vertex. It was shown in \cite{gerbner2} that $\ex^{{\mathrm col}}(n,C_4,K_2;F_2)=\cN^{{\mathrm col}}(C_4,K_2;G)$ where $G$ is obtained from blue $K_{\lfloor n/2\rfloor,\lceil n/2\rceil}$ by adding an arbitrary red edge. Observe that the red edge is not in any copy of $C_4$. This implies that $\lfloor n/2\rfloor\lceil n/2\rceil(\lfloor n/2\rfloor-1)(\lceil n/2\rceil-1)\le\rb(n,C_4,F_2)\le \lfloor n/2\rfloor\lceil n/2\rceil(\lfloor n/2\rfloor-1)(\lceil n/2\rceil-1)+1$. One can see that the proof in \cite{gerbner2} implies the following weak stability: if an $F_2$-free $n$-vertex graph $G'$ has $\cN^{{\mathrm col}}(C_4,K_2;G')=\ex^{{\mathrm col}}(n,C_4,K_2;F_2)$, then $G'=G$. This is enough for us to obtain the exact result $\rb(n,C_4,F_2)=\ex(n,C_4,F_2)=\cN(C_4,K_{\lfloor n/2\rfloor,\lceil n/2\rceil})=\lfloor n/2\rfloor\lceil n/2\rceil(\lfloor n/2\rfloor-1)(\lceil n/2\rceil-1)$.

\section{The symmetric case}

Note that the result proved in the previous sections only imply the simple bound $\rb(n,F,F)\le \ex(n,F)$ in the symmetric case (Lemma \ref{celeb}). In particular $\rb(n,F,F)=O(n^2)$, and for bipartite graphs $F$ we have $\rb(n,F,F)=o(n^2)$. We show that for non-bipartite graphs $\rb(n,F,F)=\Theta(n^2)$. First we show that $\rb(n,F,F)$ cannot significantly decrease by considering a larger graph.

\begin{proposition}\label{reszgrafszimm} If $F$ is a subgraph of $F'$, then
    $\rb(n,F',F')\le \rb(n-|V(F')|+|V(F)|,F,F)$.
\end{proposition}

\begin{proof}
    Let $F_0$ denote a graph obtained by removing a copy of $F$ from $F'$. We take a collection $\cH_0$ of $\rb(n-|V(F')|+|V(F)|,F,F)$ copies of $F$ on $n-|V(F)'|+|V(F)|$ vertices without a rainbow $F$, and we take a copy of $F_0$. We let $\cH$ consist of the copies of $F'$ that contain the copy of $F_0$ and an element of $\cH_0$. A rainbow copy of $F'$ may contain any set of vertices from the copy of $F_0$, but the remaining part must contain a rainbow $F$ on the other $n-|V(F)'|+|V(F)|$ vertices. Such an $F$ does not exist, thus $\cH$ is rainbow $F_0$-free, completing the proof.
\end{proof}

\begin{proposition}
    If $F$ is not bipartite, then $\rb(n,F,F)=\Theta(n^2)$.
\end{proposition}

\begin{proof}
    By Proposition \ref{reszgrafszimm}, it is enough to prove the statement for odd cycles. Let $k\ge 1$. We take vertices $u_i^j,v_i^j$ for $i\le k$ and $j\le n/4k$, and vertices $w_\ell$ for $\ell\le n-2k\lfloor n/4k\rfloor$. We let $\cH$ consist of the copies of $C_{2k+1}$ of the form $w_\ell u_k^ju_{k-1}^j\dots u_1^jv_1^jv_2^j\dots v_k^jw_\ell$ for each $j$ and $\ell$. Let $G$ be the graph containing the edges in the elements of $\cH$.
    
    We claim that the resulting quadratic many elements of $\cH$ do not contain a rainbow $C_{2k+1}$. Observe that a rainbow $C_{2k+1}$ contains a $w_\ell$, since those vertices cut $G$ into components of order $2k$. Similarly, a rainbow $C_{2k+1}$ contains an edge $u_1^jv_1^j$ for some $j$, as deleting those edges we obtain a bipartite graph. Then the rainbow $C_{2k+1}$ has to contain $u_i^j$ and $v_i^j$ for every $i\le k$, and in particular the edges $w_\ell u_k$ and $w_\ell v_k$. But only one element of $\cH$ contains those two edges, a contradiction completing the proof.
\end{proof}

We conjecture that $\rb(n,F,F)$ and $\ex(n,F)$ have the same order of magnitude even for bipartite graphs if they are connected.

\begin{conjecture}
   For any connected graph $F$, we have $\rb(n,F,F)=\Theta(n,F,F)$.
\end{conjecture}

Note that the above conjecture does not hold for some disconnected graphs. Let $M_2$ denote the matching of size 2. It is easy to see that $\ex(n,M_2)=n-1$ while $\rb(n,M_2,M_2)=3$ for $n\ge 4$.

We can prove the above conjecture for a large class of bipartite graphs.

\begin{proposition}
    Let $F$ be a bipartite graph such that for any homomorphism of $F$, if $u$ and $v$ are mapped to the same vertex, then they have a common neighbor. Then $\rb(n,F,F)\ge \ex(\lfloor n/|V(F)|\rfloor,F)/2$. 
\end{proposition}

\begin{proof}
    Consider an $F$-free graph $G$ on $\lfloor n/|V(F)|\rfloor$ vertices, with $\ex(\lfloor n/|V(F)|\rfloor,F)$ edges. It is well-known and easy to see that there is a bipartite subgraph $G'$ of $G$ with at least $\ex(\lfloor n/|V(F)|\rfloor,F)/2$ edges. Indeed, we consider a bipartite subgraph with the most edges, then every vertex $v$ has at least as many neighbors in the other part as in its part, thus at least half the edges incident to $v$ are in $G'$. This implies that $|E(G')|\ge |E(G)|/2$.
    Let $A$ and $B$ denote the two parts of $G'$.

    Consider a bipartition of $F$ into two independent sets and let $s$ and $t$ be the order of those two sets. We replace each vertex $u$ of $A$ by $s$ vertices $u_1,\dots,u_s$, and replace each vertex $v$ of $B$ by $t$ vertices $v_1,\dots,v_t$. For each edge $uv$ of $G'$, we place a copy of $F$ on $u_1,\dots,u_s,v_1,\dots,v_s$ such that each edge is of the form $u_iv_j$. This way we created at least $\ex(\lfloor n/|V(F)|\rfloor,F)/2$ copies of $F$ on at most $n$ vertices.

    We claim that there is no rainbow copy of $F$ here. Assume otherwise, and observe that by mapping each vertex $u_i$ of this copy to the vertex $u$ of $G'$ (where $u_i$ is one of the vertices replacing $u$), we obtain a homomorphism of $F$ to $G'$. Since $G'$ is $F$-free, at least two vertices $u$ and $v$ of $F$ are mapped into the same vertex of $G'$. Then $u$ and $v$ have a common neighbor $w$. Observe that the edges $uw$ and $vw$ appear only in the same copy of $F$, a contradiction completing the proof.
\end{proof}

Let $P_k$ denote the path on $k$ vertices.

\begin{thm}
\begin{displaymath}
\rb(n,P_4, P_4)=
\left\{ \begin{array}{l l}
n-3 & \textrm{if\/ $n\ge 5$},\\
3 & \textrm{if\/ $n=4$},\\
0 & \textrm{if\/ $n\le 3$}.\\
\end{array}
\right.
\end{displaymath}
\end{thm}

\begin{proof}
    The lower bound is given by the following construction. We take vertices $a,b,c_1,\dots,c_{n-3},d$ and paths $abc_id$. Assume that this graph contains a rainbow $P_4$. A vertex $c_i$ or $a$ can only be at one of the ends of this path as the edges incident to $c_i$ are of the same color and only one edge is incident to $a$. Therefore, $b$ and $d$ are the middle vertices, but they are not adjacent, a contradiction. In the case $n=4$, obviously any two copies of $P_4$ avoid a rainbow $P_4$.

To prove the upper bound, consider a family $\cH$ of copies of $P_4$ without a rainbow $P_4$, and let $G$ be the graph consisting of the edges appearing in elements of $\cH$. We say that an edge of $G$ is \textit{light} if it is contained in exactly one copy of $P_4$. Observe that for three different elements $H,H',H''\in\cH$, we cannot have that $H$ contains an edge of $H''$ and $H'$ contains another edge of $H''$. Indeed, those two edges have a different color and the third edge gets the color of $H''$, showing that $H''$ is rainbow.
This implies that there is at least one light edge in every element of $\cH$, as a different copy of $P_4$ cannot contain the same three edges.

Let us pick a light edge from each element of $\cH$ and let $G'$ be the resulting graph. Obviously, $G'$ is $P_4$-free. We claim that $G'$ is also triangle-free. Indeed, if $uvw$ is a triangle in $G'$, then the $P_4$ corresponding to $uv$ has a third vertex $w'$ adjacent to, say, $u$. This means that $uw'$ is an edge of $G$. Observe that $w\neq w'$ since $uw$ is a light edge. Then the edges $vw$, $wu$, $uw'$ form a rainbow $P_4$ in $G$, a contradiction. Therefore, each component of $G'$ is a star. 

\begin{clm} If $|E(G')|\ge n-2$ and $n\ge 5$, then
there is a vertex contained in at most one element of $\cH$.
\end{clm}

\begin{proof}

Let us consider a star component of $G'$ with center $u$ and leaves $v_1,\dots,v_k$ with $k\ge 4$. Assume that $uv_1$ is contained in $H\in \cH$ and there is $H'\in\cH$ with $H'\neq H$ that contains $v_1$. Without loss of generality, $H'$ does not contain $v_2$. Then we can go from $v_1$ to a vertex $w\neq u$ using $H'$. We can go from $v_1$ to $u$ using $H$, and from $v_2$ to $u$ using the single element of $\cH$ containing $uv_2$. By the same argument, for $k\ge 2$, we obtain a contradiction if $H'$ has a vertex not in the star. 

We obtained that either we have a single star component with $k=3$, thus $n=4$, or each component has at most two vertices, in which case there are at least three components, thus at most $n-3$ edges in $G'$, a contradiction.
\end{proof}

Let us return to the proof of the theorem. We apply induction on $n$. Let us start with the base cases $n\le 5$. The cases $n\le 3$ are trivial. If $n=4$, we cannot have three edges in $G'$ without a triangle or $P_4$. In the case $n=5$, $G'$ has to consist of a two-edge path $uvw$ and an isolated edge $u'v'$. Let $H$ be the copy of $P_4$ containing the edge $u'v'$. If an edge of $H$ from $u'$ or $v'$ goes to $u$ or $w$, then it forms a rainbow $P_4$ with $uv$ and $vw$. Therefore, the edge of $H$ adjacent to $u'v'$ goes to $v$. Without loss of generality, $v'v$ is an edge of $H$. Then the third edge of $H$ cannot be incident to $v$, as $uv$ and $vw$ are light edges, $vv'$ is already in $H$ and $vu'$ would create a triangle. Thus, the third edge of $H$ goes from $u'$ to $u$ or $w$, a possibility we have already ruled out. This contradiction proves the case $n=5$.

Now consider an arbitrary $n\ge 6$. Let $v$ be a vertex contained in at most one element of $\cH$ and delete that vertex. Then we obtain a family $|\cH'|$ of copies of $P_4$ on $n-1$ elements such that $|\cH|\le |\cH'|+1$. By induction, 
$|\cH'|\le n-4$, completing the proof.
\end{proof}

%$B_2$ vs $B_2$ $n^2/12+o$ mert olyan ahol 4 ponton több van, az kevés mert $B_t$-freeben háromszögek, olyan amiben mindkét háromszög 2-heavy, az kevés van ugyanezért, és abban kevésszer vannak fedve a háromszögek. Amiben csak egy háromszög 2-heavy, és az 3-heavy, abban van három light él, ahol meg 2 és 1, azokra 3-éleket számolunk, kéne a Berge multihipergráfokra is...

Recall that $B_t$ denotes the \textit{book graph}, which consists of $t$ triangles sharing an edge $uv$. We call $u$ and $v$ the \textit{rootlet vertices} and the other $t$ vertices are the \textit{page vertices}. Alon and Shikhelman \cite{as} showed that $\ex(n,K_3,B_t)=o(n^2)$. The author \cite{gerbner} showed that $\ex(n,K_r,B_t)=o(n^2)$. Here we give another simple bound.

\begin{proposition}For any $r$ and $t$, we have $\ex(n,B_t,B_r)=o(n^2)$.\end{proposition}

\begin{proof} Let $G$ be a $B_r$-free graph on $n$ vertices. We pick copies of $B_t$ the following way. We pick a triangle $uvw$, $o(n^2)$ ways, than we pick $t-1$ other common neighbors of $u$ and $v$. As $u$ and $v$ have less than $r$ common neighbors, there are $O(1)$ ways to pick the additional vertices. Clearly every copy of $B_t$ is counted at least once, completing the proof.
\end{proof}

The author determined $\ex_3(n,\textup{Berge-}B_r)=\rb(n,K_3,B_r)$ exactly for $n$ sufficiently large. Its value is $\lfloor n^2/8\rfloor$ if $r\le 2$ and $\lfloor n^2/8\rfloor+(r-1)^2$ for $r\ge 3$.
Here we determine the asymptotics of $\rb(n,B_t,B_r)$ for every $r\ge t$.

\begin{thm} %$\rb(n,B_t,B_t)=(1+o(1))n^2/8$.
   If $r\ge t\ge 2$, then we have $\rb(n,B_t,B_r)=(1+o(1))n^2/8$.
\end{thm}

\begin{proof}
    To prove the lower bound, let us take vertices $u_iv_i$ for $i\le n/4$, $w_1,\dots,w_{t-1}$ and $x_j$ for $j\le n-2\lfloor n/4\rfloor-t+1$. For every $i$ and $j$, we take the copies of $B_t$ where the rootlet vertices are $u_i$ and $v_i$ and the page vertices are $w_1,\dots,w_{t-1}$ and $x_j$.

    The subgraph obtained by deleting $w_1,\dots,w_{t-1}$ is rainbow triangle-free (this is the construction from \cite{gyori} and \cite{gh}). Therefore, in any rainbow $B_r$, one of the rootlet vertices is $w_\ell$. As the only neighbors of $w_\ell$ are $u_i$ and $v_i$, the other rootlet vertex is, say, $u_1$. The only common neighbor of $w_\ell$ and $u_1$ is $v_1$, a contradiction.

    Let us continue with the upper bound. Let $\cH$ be a family of copies of $B_t$ on $n$ vertices. We say that a triangle or edge is \textit{$p$-heavy} if it is contained in at least $p$ elements of $\cH$ and \textit{$p$-light} otherwise.
    %Let $\cH_1$ denote the subfamily of copies of $B_2$ that contain only $(2t+1)$-heavy triangles. Then each edge in the elements of $\cH_1$ is $(2t+1)$-heavy. Let $G_1$ denote the graph having the edges in the elements of $\cH_1$. By Proposition \ref{} $G_1$ is $B_t$-free, thus $G_1$ contains $o(n^2)$ copies of
    Observe that by Proposition \ref{gree}, every element of $\cH$ contains a $(2r+1)$-light edge, thus a $(2r+1)$-light triangle.

    Let $\cH_1$ denote the subfamily of $\cH$ consisting of the copies of $B_t$ that contain a 2-heavy $(2r+1)$-light triangle.

    \begin{clm}
        $|\cH_1|=o(n^2)$.
    \end{clm}
    \begin{proof}[Proof of Claim]
Let us pick a 2-heavy $(2r+1)$-light triangle from each element of $\cH_1$ and let $G_1$ denote the graph consisting of the edges of those triangles. Assume that $G_1$ contains a $B_{2tr}$ with rootlet vertices $u,v$ and page vertices $w_1,\dots,w_{2tr}$. As each edge is 2-heavy, we can take two elements of $\cH_1$ to form a rainbow path with edges $uw_i$ and $w_iv$. Those two elements contain at most $2t+1$ other vertices.

We pick an element of $\cH_1$ that contains the edge $uv$. It avoids without loss of generality $w_1,\dots,w_{2tr-t}$. Then we pick two elements of $\cH_1$ to form a rainbow path $uw_1v_1$, they avoid without loss of generality $w_2,\dots,w_{2tr-3t-1}$. We continue this way, the elements of $\cH_1$ we pick for the path $uw_iv$ avoid $w_{i+1},\dots,w_{2tr-i(2t+1)-t}$, hence we can pick elements for the path $uw_tv$, obtaining a rainbow $B_r$, a contradiction. 

We obtained that $G_1$ is $B_{2tr}$-free, thus contain $o(n^2)$ triangles. For each element of $\cH_1$, we picked a triangle in $G_1$, and each triangle was picked at most $2r+1$ times, completing the proof.
    \end{proof}

    Let us return to the proof of the theorem. It is left to deal with the elements of $\cH$ that contain a 2-light triangle. We pick such a triangle for each of them. Those triangles clearly cannot contain a rainbow $B_r$, thus there are at most $\rb(n,K_3,B_r)=n^2/8+O(1)$ of them, completing the proof.
\end{proof}

    \section{Concluding remarks}

    We have proved generalizations of some fundamental results concerning the Tur\'an number of Berge hypergraphs. Probably there are several other results that can be generalized the same way.
 We also mention some variants of Berge hypergraphs that can be studied in our setting. 
 
 The $t$-wise Berge hypergraphs \cite{gnpv} are those where we take $t$ hyperedges for every edge. In our setting we look for graphs $F$ such that we can pick $t$ distinct copies of $H$ for every edge of $F$ (altogether $t|E(F)|$ elements of $\cH$).

    A natural idea is to study the same problem for hypergraphs instead of graphs. This was already started in the Berge setting, see \cite{bgkkp}.

    Another natural idea is to consider families $\cH$ that consists of different graphs. This was also studied in the Berge setting, as taking cliques of different order corresponds to non-uniform hypergraphs.

We have already mentioned in Section 2 the expansion, which is a specific Berge copy of a graph. In our setting, this corresponds to finding a rainbow copy of $F$ where vertices in the elements of $\cH$ corresponding to the edges of $F$ are distinct outside the necessary intersections in $F$.

\bigskip

\textbf{Funding}: Research supported by the National Research, Development and Innovation Office - NKFIH under the grants KH 130371, SNN 129364, FK 132060, and KKP-133819.

\end{document}